\documentclass{article}
\usepackage[a4paper]{geometry}
\usepackage[initials]{amsrefs}
\usepackage{amssymb,amsthm,color,enumitem,mathrsfs,mleftright,tikz-cd,hyperref}
\usepackage[fleqn]{amsmath}
\usepackage{amsmath}
\usepackage{color}
\usepackage{thmtools}
\usepackage{thm-restate}

\newcommand{\edit}[1]{#1}

\newtheorem{theorem}{Theorem}[section]
\newtheorem{lemma}[theorem]{Lemma}
\newtheorem{corollary}[theorem]{Corollary}
\newtheorem{proposition}[theorem]{Proposition}

\theoremstyle{definition}
\newtheorem{definition}[theorem]{Definition}
\newtheorem{remark}[theorem]{Remark}

\theoremstyle{plain}

\tikzcdset{arrow style=tikz, diagrams={>=stealth}}
\tikzcdset{every label/.append style = {font = \normalsize}}
\definecolor{hcol}{RGB}{33,120,33}
\definecolor{fcol}{RGB}{44,90,160}
\definecolor{gcol}{RGB}{170,0,0}

\numberwithin{equation}{section}
\numberwithin{figure}{section}

\renewcommand{\geq}{\geqslant}
\renewcommand{\leq}{\leqslant}
\newcommand{\C}{\mathbb{C}}
\newcommand{\D}{\mathbb{D}}
\newcommand{\R}{\mathbb{R}}

\makeatletter
\renewcommand\section{\@startsection {section}{1}{\z@}%
                                   {-3.5ex \@plus -1ex \@minus -.2ex}%
                                   {1.3ex \@plus.2ex}%
                                   {\bf\large}}
\makeatother

\makeatletter
\renewcommand\subsection{\@startsection {subsection}{1}{\z@}%
                                   {-3.5ex \@plus -1ex \@minus -.2ex}%
                                   {1.3ex \@plus.2ex}%
                                   {\bf}}
\makeatother

\title{ \vspace{-5ex}\bf \large Iterated function systems of holomorphic maps\footnotetext{
		2020 Mathematics Subject Classification: Primary 30D05; Secondary 37F99.
		
		Key words: holomorphic map, hyperbolic Riemann surface, iterated function system. 
		
		The first author was partially supported by the French-Italian University and by Campus France through the Galileo program, under the project \textit{From rational to transcendental: complex dynamics and parameter spaces}, as well as by Istituto Nazionale di Alta Matematica (INdAM).
		
		The second author was supported by EPSRC grant EP/W002817/1.
		
		There is no data associated with this article.}}

\author{\normalsize Marco Abate and Ian Short}

\numberwithin{equation}{section}
\numberwithin{figure}{section}

\renewcommand{\ge}{\geqslant}
\renewcommand{\le}{\leqslant}
\renewcommand{\geq}{\geqslant}
\renewcommand{\leq}{\leqslant}

\DeclareMathOperator{\id}{\textnormal{id}}

\DeclareMathOperator{\Hol}{\textnormal{Hol}}

\DeclareMathOperator{\Aut}{\textnormal{Aut}}

\setlist[enumerate]{leftmargin=25pt,itemsep=0pt,topsep=3pt}
\setlist[enumerate,1]{label=\textnormal{(\roman*)}}

\date{\normalsize\today}

\begin{document}

\maketitle

\begin{abstract}
We unify and advance a host of works on iterated function systems of holomorphic self-maps of hyperbolic Riemann surfaces. Our foremost result is a generalisation to left iterated function systems of an unpublished and little known theorem of Heins on iteration in the unit disc. Applications abound -- to work of Benini et al.\ on transcendental dynamics, to the theory of hyperbolic steps of holomorphic maps, and to left semiconjugacy in the unit disc. We extend other work of Benini et al.\ and Ferreira on relatively compact left iterated function systems, and we prove a hyperbolic distance inequality for holomorphic maps that generalises a theorem of Bracci, Kraus, and Roth. Additionally, we  strengthen results of the first author and Christodoulou on left iterated function systems, removing the need for Bloch domains, and we answer an open question from their work. Finally, we establish a version of the Heins theorem for right iterated functions systems, and we generalise theorems of Beardon and Kuznetsov on right iterated function systems in relatively compact semigroups of holomorphic maps.   
\end{abstract}

\section{Introduction}\label{section1}

The objective of this paper is to explore the dynamics of iterated function systems of holomorphic self-maps of hyperbolic Riemann surfaces. In so doing we advance the results of a host of other works in this field including \cites{AbCh2022,BeCaMiNg2004,Fe2023,He1991,Ku2007}. Throughout, we let $X$ denote a hyperbolic Riemann surface with \edit{hyperbolic} distance $\omega_X$. We denote the open unit disc by $\mathbb{D}$ and write $\omega$ for $\omega_\mathbb{D}$.

Let $C(X,X)$ denote the space of continuous self-maps of $X$ endowed with the compact-open topology. In this topological space, a sequence $\{f_n\}$ converges to a map $f$ if and only if $f_n(z)\to f(z)$ uniformly on compact subsets of $X$, in which case we write $f_n\to f$. The space $\Hol(X,X)$ of holomorphic self-maps of $X$ is a closed subspace of $C(X,X)$, and the space $\Aut(X)$ of conformal automorphisms of $X$ is a closed subspace of $\Hol(X,X)$ (see \cite{Ab2022}*{Corollary~1.7.21}). A \emph{semicontraction} of $X$ is a map $f\colon X\longrightarrow X$ with 
\[
\omega_X(f(z),f(w))\leq \omega_X(z,w),\quad\text{for $z,w\in X$.}
\]
By the Schwarz--Pick lemma, each holomorphic self-map of $X$ is a semicontraction of $X$, and automorphisms of $X$ are isometries in the hyperbolic metric. 

A \emph{left iterated function system} in $\Hol(X,X)$ is a sequence $\{L_n\}$ given by $L_n=f_n\circ f_{n-1}\circ \dots\circ f_1$, where $f_n\in \Hol(X,X)$. We say that $\{L_n\}$ is \emph{generated} by $\{f_n\}$. A \emph{right iterated function system} in $\Hol(X,X)$  is a sequence $\{R_n\}$ of the form $R_n=f_1\circ f_{2}\circ \dots\circ f_n$, where $f_n\in \Hol(X,X)$. We assume throughout that any sequence $\{h_n\}$ in $\Hol(X,X)$ comes with a zeroth term $h_0$ equal to the identity map $\id_X$, unless stated otherwise.

Iterated function systems originated in fractal geometry as systems of contractions of complete metric spaces \edit{-- see  \cite{Ba1988,BaDe1985,Hu1981} for some foundational works on this subject}. The language of iterated function systems was used for holomorphic self-maps of hyperbolic Riemann surfaces in, for example, \cite{AbCh2022,KeLa2007}. Such systems have appeared under various other names; for a quick sample, they are described as non-autonomous dynamical systems in \cite{BeEvFaRiSt2024}, where the qualifiers left/right are replaced by forwards/backwards (as they are in \cite{KeLa2007} and other works). Beardon refers to iterated function systems as composition sequences in \cite{Be2001}, and this phrase was adopted by Jacques and the second author in \cite{JaSh2022}. Gou\"ezel and Karlsson use the language of cocycles in \cite{GoKa2020}, where they establish a version of the Wolff--Denjoy theorem for random cocycles of semicontractions of a metric space.

The dynamics of left and right iterated function systems are quite different, and working with each type of sequence has its own distinct advantages. For instance, for a left iterated function system $\{L_n\}$, we have
\[
\omega_X(z,w) \geq \omega_X(L_1(z),L_1(w)) \geq \omega_X(L_2(z),L_2(w))\geq \dotsb,
\]
for any pair of points $z,w\in X$. This property fails for right iterated function systems, in general. However, any right iterated function system $\{R_n\}$ satisfies other advantageous properties; for example,  the inclusion $f_n(X)\subseteq X$ implies that $R_n(X)\subseteq R_{n-1}(X)$, and hence we obtain a nested sequence of sets
\[
X \supseteq R_1(X)\supseteq R_2(X) \supseteq \dotsb.
\]
Also, for right iterated function systems, we have
\[
\omega_X(R_{n-1}(z),R_{n}(z))\leq \omega_X(z,f_n(z)),
\] 
for $z\in X$. When the maps $\{f_n\}$ are chosen from a finite collection, this inequality implies that there is a bounded step between successive terms of the sequence $\{R_n(z)\}$, giving us tight control on the dynamics of $\{R_n\}$. 

Let us now summarise our main results (labelled A to~H) for left and right iterated function systems in turn.

\subsection{Left iterated function systems}

Our first result concerns left iterated function systems on the unit disc.

\begin{restatable}{thm}{theoremA}
\label{theoremA}
For any left iterated function system $L_n=f_n\circ f_{n-1}\circ \dots\circ f_1$ in $\Hol(\D,\D)$ there is a sequence $\{\gamma_n\}$ of conformal automorphisms of~$\D$ and a map $h\in\Hol(\D,\D)$, unique up to left composition by elements of $\Aut(\D)$, such that $\gamma_n^{-1}\circ L_n \to h$.
\end{restatable}

Theorem~\ref{theoremA} is a generalisation of an unpublished result \cite{He1991}*{Theorem~1} of Heins, which is a similar statement but with the assumption that all the maps $f_n$ are equal. 

\edit{Using hyperbolic distortion we can obtain a corollary of Theorem~\ref{theoremA} in which the sequence $\{\gamma_n\}$ does not feature.} Briefly, the \emph{hyperbolic distortion} of a self-map $f$ of a hyperbolic Riemann surface $X$ is the function 
\[
f^\#(z) = \lim_{w\to z} \frac{\omega_X\bigl(f(z),f(w)\bigr)}{\omega_X(z,w)}\;.
\]
With this notation we have $(\gamma_n^{-1}\circ L_n)^\#=((\gamma_n^{-1})^\#\circ L_n)L_n^\#=L_n^\#$, since the hyperbolic distortion of an automorphism is identically~1. We discuss hyperbolic distortion in detail in Section~\ref{section3}.

\begin{restatable}{cor}{corollaryB}
\label{corollaryB}
For any left iterated function system $L_n=f_n\circ f_{n-1}\circ \dots\circ f_1$ in $\Hol(\D,\D)$ there is a map $h\in\Hol(\D,\D)$, unique up to left composition by elements of $\Aut(\D)$, such that $L_n^\# \to h^\#$.
\end{restatable}

Using hyperbolic distances gives us another perspective on Theorem~\ref{theoremA}. The automorphisms $\gamma_n$ are isometries, so we have $\omega\bigl(L_n(z),\gamma_n\bigl(h(z)\bigr)\bigr)\to 0$, for $z\in\D$; this shows us that $\{L_n\}$ behaves asymptotically like $\{\gamma_n\circ h\}$. Also, since $\omega\bigl(\gamma_n^{-1}\circ L_n(z),\gamma_n^{-1}\circ L_n(w)\bigr)=\omega\bigl(L_n(z),L_n(w)\bigr)$, we obtain the following corollary of Theorem~\ref{theoremA}.

\begin{restatable}{cor}{corollaryC}
\label{corollaryC}
For any  left iterated function system $L_n=f_n\circ f_{n-1}\circ \dots\circ f_1$ in $\Hol(\D,\D)$ there is a map $h\in\Hol(\D,\D)$, unique up to left composition by elements of $\Aut(\D)$, such that 
$\omega\bigl(L_n(z),L_n(w)\bigr)\to \omega\bigl(h(z),h(w)\bigr)$, for all $z$,~$w\in\D$.
\end{restatable}

Special cases of Corollaries~\ref{corollaryB} and~\ref{corollaryC} when all the maps $f_n$ are equal can be found in Heins's unpublished work \cite{He1991}*{Theorems~2 and~3}.

Corollary~\ref{corollaryC} has applications in transcendental dynamics. For example, in \cite{BeEvFaRiSt2022} Benini et al.\ consider the behaviour of the iterates $\{f^n\}$ of some transcendental entire function $f$ on a sequence $\{U_n\}$ of simply-connected wandering domains, where $f(U_{n-1})\subseteq U_n$. They study the quantity 
\[
c(z,z')= \lim_{n\to\infty} \omega_{U_n}\bigl(f^n(z),f^n(z')\bigr),\quad \text{for $z,z'\in U_0$}
\]
(see, in particular, \cite{BeEvFaRiSt2022}*{Theorem~A}). Let $\mu_n$ be a one-to-one conformal map from $U_n$ onto \edit{$U_0$}, and let $f_n=\mu_n\circ f\circ \mu_{n-1}^{-1}$ (with $\mu_0$ the identity map). Then $L_n=f_n\circ f_{n-1}\circ \dots\circ f_1$ is a left iterated function system in \edit{$\Hol(U_0,U_0)$} and $L_n=\mu_n\circ f^n$. Corollary~\ref{corollaryC} \edit{(with $U_0$ in place of $\mathbb{D}$)} tells us that there is \edit{$h\in \Hol(U_0,U_0)$} with $\omega\bigl(L_n(z),L_n(z')\bigr)\to \omega\bigl(h(z),h(z')\bigr)$; hence
\[
c(z,z')=\omega\bigl(h(z),h(z')\bigr)\;.
\]
That is, the limiting quantity $c(z,z')$ studied in \cite{BeEvFaRiSt2022} and subsequent works such as \cites{BeEvFaRiSt2024,Fe2022}  is realised by some holomorphic function $h$, and whether $c$ is zero or not zero corresponds to whether $h$ is a constant function or otherwise. We shall  discuss this further in Sections~\ref{section4} and~\ref{section5}.

Corollary~\ref{corollaryC} also has applications to the theory of iteration of a single holomorphic self-map of $\D$. Given $f\in\Hol(\D,\D)$, we observe that $L_n=f^n$ is a left iterated function system (and a right iterated function system), so we can find $h\in \Hol(\D,\D)$ with 
\[
\omega\bigl(f^{n}(z),f^{n+1}(z)\bigr)=\omega\bigl(f^{n}(z),f^{n}\bigl(f(z)\bigr)\bigr)\to \omega\bigl(h(z),h\bigl(f(z)\bigr)\bigr)\;.
\]
Now, the limiting value of the sequence with $n$-th term $\omega\bigl(f^{n}(z),f^{n+1}(z)\bigr)$ is called the \emph{hyperbolic step} of $f$ at $z$, and the theory of this quantity is explored in \cite{Ab2022}*{Section 4.6} (see also references therein). Corollary~\ref{corollaryC} demonstrates that the hyperbolic step can be realised explicitly using the holomorphic map $h$, and whether the hyperbolic step is zero or not zero corresponds to whether or not $h$ is constant, as we shall see in detail in Section~\ref{section4}. 

We include a third and final corollary of Theorem~\ref{theoremA}. Given maps $f,g \in\Hol(X,X)$ we say that $f$ is \emph{left semiconjugate} in $\Hol(X,X)$ to $g$ if there is a nonconstant map $\phi\in \Hol(X,X)$ with $\phi\circ f=g\circ \phi$. In the following statement we denote by $h$ the limit function of the sequence $L_n=f^n$ from Theorem~\ref{theoremA}.

\begin{restatable}{cor}{corollaryD}
\label{corollaryD}
A holomorphic map $f\in\Hol(\D,\D)$ is left semiconjugate in $\Hol(\D,\D)$ to a conformal automorphism of $\D$ if and only if $h$ is not constant.
\end{restatable}

Corollary~\ref{corollaryD} is inspired by the unpublished result \cite{He1991}*{Theorem~4}. Theorem~\ref{theoremA} and Corollary~\ref{corollaryD} will be proved in Section~\ref{section4}.

Notice that Theorem~\ref{theoremA} fails for hyperbolic Riemann surfaces in general. To see this, consider any hyperbolic Riemann surface with trivial automorphism group and choose the maps $f_n$ to be any constant functions such that $\{L_n\}$ does not converge in $\Hol(X,X)$. With only a little more care one can find similar examples in which none of the $f_n$ are constant functions.

We move on to consider further convergence criteria for left iterated function systems. Theorem~2.1 of \cite{BeEvFaRiSt2022} --- which was generalised by Ferreira in \cite{Fe2023}*{Theorem~1.1} --- says that with the hypotheses that each function $f_n\in\Hol(\D,\D)$ fixes $0$ and that $L_n=f_n\circ f_{n-1}\circ \dots\circ  f_1$ converges in $\Hol(\D,\D)$ to some function $F$, we have that $F$ is nonconstant if and only if $\sum (1-|f_n'(0)|)<+\infty$. (Ferreira's theorem also shows that if the $f_n$ are inner functions and $F$ is not constant then it also is an inner function; we will not go into this here.) Our next theorem generalises these results.

\begin{restatable}{thm}{theoremE}
\label{theoremE}
Let $X$ be a hyperbolic Riemann surface and let $L_n=f_n\circ f_{n-1}\circ \dots\circ f_1$ be a left iterated function system that is relatively compact in $\Hol(X,X)$. Suppose that the maps $f_n$ are nonconstant. Then the following statements are equivalent.
\begin{enumerate}
\item All limit points of $\{L_n\}$ in $\Hol(X,X)$ are nonconstant.
\item The sequence $\{L_n\}$ has a nonconstant limit point in $\Hol(X,X)$.
\item There exists $z_0\in X$ with $\sum_n \bigl(1-f_n^\#(z_0)\bigr)<+\infty$.
\item For all $z\in X$ we have $\sum_n \bigl(1-f_n^\#(z)\bigr)<+\infty$.
\end{enumerate}
\end{restatable}

This is more general than \cite{BeEvFaRiSt2022}*{Theorem~2.1} in that it applies to all hyperbolic Riemann surfaces, the assumption that the maps $f_n$ fix $0$ has been weakened (since that assumption implies that $\{L_n\}$ is relatively compact), and it is no longer assumed that $\{L_n\}$ converges. Theorem~\ref{theoremE} will be proved in Section~\ref{section5}. \edit{Then in Section~\ref{section6} we prove the following inequality (and its corollary).}

\begin{restatable}{thm}{theoremF}
\label{theoremF}
Let $f\in \Hol(\D,\D)$ and let $w\in\D$. Then there exists  $\gamma\in \Aut(\D)$ such that
\[
\omega\bigl(f(z),\gamma(z)\bigr)\leq 2e^{4\omega(z,w)}\bigl(1-f^\#(w)\bigr),
\]
for all $z\in\D$.
\end{restatable}

A corollary of Theorem~\ref{theoremF} is the following recent theorem of Bracci, Kraus, and Roth \cite{BrKrRo2023}*{Theorem~2.1}.

\begin{restatable}{cor}{corollaryG}
\label{corollaryG}
Let $f\in\Hol(\D,\D)$ be such that
\[
f^\#(z_n) = 1+o\bigl((1-|z_n|)^2\bigr)
\]
for some sequence $\{z_n\}$ in $\D$ with $|z_n|\to 1$. Then $f\in \Aut(\D)$ and hence $f^\#(z)=1$ for $z\in\D$.
\end{restatable}

Bracci, Kraus, and Roth use this result to prove the Burns--Krantz theorem and as a first step to other results; see \cite{BrKrRo2023} itself or \cite{Ab2022}*{Section~2.7}. 

Our final pair of theorems on left iterated functions systems together generalise \cite{AbCh2022}*{Theorem~1.5} by the first author and Christodoulou. To state that theorem, we recall that a \emph{Bloch domain} $\Omega$ in a hyperbolic Riemann surface $X$ is a subdomain of $X$ with the property that there is a uniform bound on the radii of any hyperbolic discs in $X$ that lie in $\Omega$. The first author and Christodoulou proved that if $f_n$ are holomorphic maps from $X$ into a Bloch domain $\Omega$, and if $a_n$ is the unique fixed point of $f_n$ in $\Omega$, then the left iterated function system $L_n=f_n\circ f_{n-1}\circ \dots \circ f_1$ converges to a constant $a$ in $X$ if and only if $a_n\to a$. (There is another lesser part to \cite{AbCh2022}*{Theorem~1.5} which we will not discuss.) 

For the next theorem, recall that $\id_X$ denotes the identity map in $\Hol(X,X)$ (and $\overline{\mathscr{F}}$ is the closure of $\mathscr{F}$ in $\Hol(X,X)$).

\begin{restatable}{thm}{theoremH}
\label{theoremH}
Let $X$ be a hyperbolic Riemann surface and let $\mathscr{F}$ be a subfamily of $\Hol(X,X)$ for which $\id_X\notin \overline{\mathscr{F}}$. Suppose that the left iterated function system $L_n=f_n\circ f_{n-1}\circ \dots \circ f_1$, where $f_n\in \mathscr{F}$, converges on $X$ to a constant $a$ in $X$. Then, for sufficiently large $n$, the map $f_n$ has a fixed point $a_n\in X$, and $a_n\to a$.
\end{restatable}

Theorem~\ref{theoremH} is a significant generalisation of one part of \cite{AbCh2022}*{Theorem~1.5}, because the strong assumption that $f_n$ maps $X$ into a Bloch domain has been replaced with the mild assumption that $\id_X\notin \overline{\mathscr{F}}$. The converse implication does not hold under such mild hypotheses; instead we have the following theorem. In this theorem we refer to an automorphism $f$ of $X$ as \emph{pseudoperiodic} if it is not periodic and the identity map $\id_X$ is a limit point of the sequence $\{f^n\}$.

\begin{restatable}{thm}{theoremI}
\label{theoremI}
Let $X$ be a hyperbolic Riemann surface and let $\mathscr{F}$ be a subfamily of $\Hol(X,X)$ for which $\overline{\mathscr{F}}$ does not contain any periodic or pseudoperiodic automorphisms. Suppose that the left iterated function system $L_n=f_n\circ f_{n-1}\circ \dots \circ f_1$, where $f_n\in \mathscr{F}$, is relatively compact in $\Hol(X,X)$.  Suppose also that each map $f_n$ has a fixed point $a_n\in X$ and that $a_n\to a\in X$.  Then $\{L_n\}$ converges to the constant map with value~$a$.
\end{restatable}

This time the Bloch domain condition has been replaced with the weaker assumptions that $\overline{\mathscr{F}}$ contains no periodic or pseudoperiodic automorphisms and that $\{L_n\}$ is relatively compact in $\Hol(X,X)$. To see that these truly are weaker assumptions, \edit{we recall from \cite{BeCaMiNg2004} that a consequence of the Bloch domain condition is that there exists $\ell\in (0,1)$ with $\omega_X\bigl(f_n(z),f_n(w)\bigr)\leq \ell \omega_X(z,w)$, for all $z,w\in X$ and $n\in\mathbb{N}$.  It follows from this that no limit function of the sequence $\{f_n\}$ is an automorphism of $X$. Also, we have}
\begin{align*}
\omega_X\bigl(a,f_n(a)\bigr) &\leq \omega_X(a,a_n)+\omega_X\bigl(a_n,f_n(a_n)\bigr)+\omega_X\bigl(f_n(a_n),f_n(a)\bigr)\\
&\leq \omega_X(a,a_n)+\ell \omega_X(a_n,a)\;,
\end{align*}
so there exists $L>0$ with $\omega_X\bigl(a,f_n(a)\bigr)<L$ for $n\in\mathbb{N}$. Then
\begin{align*}
\omega_X\bigl(f_n\circ f_{n-1} \circ \dots \circ f_{n-k+1}(a),f_n\circ f_{n-1} \circ \dots \circ f_{n-k}(a)\bigr) &\leq \ell^{k}\omega_X\bigl(a,f_{n-k}(a)\bigr)\\
&< L\ell^{k}\;,
\end{align*}
for $k=1,2,\dots,n-1$. By summing these we see that $\omega_X\bigl(L_n(a),a\bigr)<L/(1-\ell)$, for $n\in\mathbb{N}$. Therefore $\{L_n\}$ is relatively compact in $\Hol(X,X)$ (see Theorem~\ref{theorem1}, to follow), as required.

Theorems~\ref{theoremH} and~\ref{theoremI} will be proved in Section~\ref{section7}. 

In Section~\ref{section8} we provide an example that addresses an open question from \cite{AbCh2022}. This example concerns sequences of functions $\{h_n\}$ that are \emph{compactly divergent}, which means that, for any compact subset $K$ of $X$, there is a positive integer $n_0$ with $h_n(K)\cap K=\varnothing$, for $n\geq n_0$. Therorem~1.7 of \cite{AbCh2022} states that if $F$ is a holomorphic self-map of a hyperbolic Riemann surface $X$ for which the sequence of iterates $\{F^n\}$ is compactly divergent, and if $\{f_n\}$ is a sequence in $\Hol(X,X)$ that converges sufficiently quickly to~$F$, then the left iterated function system $\{L_n\}$ generated by $\{f_n\}$ is compactly divergent also. The open question from \cite{AbCh2022} is whether one can find a sequence $\{f_n\}$ in $\Hol(X,X)$ that converges to $F$ (slowly) such that $\{L_n\}$  is not compactly divergent. The example we offer is such that $\{L_n\}$ neither converges in $\Hol(X,X)$ and nor is it compactly divergent, thereby demonstrating the necessity of a control on the speed of convergence towards~$F$ in \cite{AbCh2022}*{Theorem 1.7}. We thank Marco Vergamini for a useful suggestion. 

Finally, in Section~\ref{section8} we also give another example of a wildly behaved left iterated function system. Indeed, we construct a sequence $\{\gamma_n\}$ of automorphisms of~$\D$ that converges to $\id_\D$ such that the left iterated function system generated by $\{\gamma_n\}$ is dense in~$\Aut(\D)$. 

\subsection{Right iterated function systems}

Next we present a version of Theorem~\ref{theoremA} for right rather than left iterated function systems. Thus, given a right iterated function system $R_n=f_1\circ f_{2}\circ \dots\circ f_n$ in $\Hol(\D,\D)$, we seek a sequence of automorphisms $\{\gamma_n\}$ of $\D$ and $h\in \Hol(\D,\D)$ with $R_n\circ \gamma_n^{-1}\to h$. To obtain a result of this type, it is necessary to assume the existence of a backward orbit for $\{R_n\}$; that is, we need to assume that there is a sequence $\{w_n\}$ in $\D$ with $f_n(w_n)=w_{n-1}$, for $n\in\mathbb{N}$. Without this assumption, it could be that the nested sequence of sets $\overline{\D}\supseteq \overline{R_1(\D)}\supseteq \overline{R_2(\D)}\supseteq \dotsb$ satisfies
\[
\bigcap_{n=1}^\infty \overline{R_n(\D)}\subseteq \partial \D\;.
\]
This would render the deduction $R_n\circ \gamma_n^{-1}\to h$ unobtainable because $R_n\circ \gamma_n^{-1}(\D)=R_n(\D)$.

\begin{restatable}{thm}{theoremJ}
\label{theoremJ}
Let $R_n=f_1\circ f_{2}\circ \dots\circ f_n$ be a right iterated function system in $\Hol(\D,\D)$ for which there exists an infinite backward orbit $\{w_n\}$. Then there exists a sequence $\{\gamma_n\}$ in $\Aut(\D)$ with $\gamma_n(w_n)=w_0$ and $h\in \Hol(\D,\D)$ for which $R_n\circ\gamma_n^{-1}\to h$. Furthermore, $h$ is uniquely specified by $\{R_n\}$ and $\{w_n\}$ up to right composition by elements of $\Aut(\D)$.
\end{restatable}

We prove Theorem~\ref{theoremJ} in Section~\ref{section9}. 

Our final result on right iterated function systems is a generalisation of a theorem of Kuznetsov \cite{Ku2007}, who proved the equivalence of statements (i) and (iv), below, in the special case when $X$ is a hyperbolic plane domain.

\begin{restatable}{thm}{theoremK}
\label{theoremK}
Let $R_n=f_1\circ f_{2}\circ \dots\circ f_n$ be a right iterated function system that lies in a relatively compact semigroup in $\Hol(X,X)$. Suppose that the maps $f_n$ are nonconstant.  Then the following statements are equivalent.
\begin{enumerate}
\item The sequence $\{R_n\}$ converges to a constant in $X$.
\item There exists a subsequence of $\{R_n\}$ that converges to a constant in $X$.
\item There exists $z_0\in X$ with $\sum_n (1-f_n^\#(z_0))=+\infty$.
\item For all $z\in X$ we have $\sum_n (1-f_n^\#(z))=+\infty$.
\end{enumerate}
\end{restatable}

The proof of Theorem~\ref{theoremK} is far shorter than that of \cite{Ku2007}; it is given in Section~\ref{section10}. 

\section{Relatively compact families of semicontractions}
\label{section2}

Central to our work is the following result, which is a corollary of the Arzelà–Ascoli theorem. Such is its importance that we include a proof, even though one can likely be found elsewhere. In the statement of the theorem we write $\mathscr{F}(z)$ for the set $\{f(z)\mid f\in\mathscr{F}\}$. \edit{Recall that a \emph{proper} metric space is a metric space in which all closed bounded subsets are compact. All hyperbolic Riemann surfaces are proper metric spaces.}

\begin{theorem}\label{theorem1}
Let $\mathscr{F}$ be a family of semicontractions of a \edit{proper} metric space~$X$. Then $\mathscr{F}$ is relatively compact in $C(X,X)$ if and only for some (and hence any) $z\in X$ the set $\mathscr{F}(z)$ is bounded.
\end{theorem}

\begin{proof}
\edit{Since $X$ is proper it must be locally compact and complete. As a consequence, we can apply the} Ascoli--Arzel\`a theorem (see, for example, \cite{Ke1955}), which tells us that a family $\mathscr{G}\subseteq C(X,X)$ is relatively compact in $C(X,X)$ if and only if 
\begin{enumerate}
\item $\mathscr{G}$ is equicontinuous and 
\item $\mathscr{G}(z)$ is relatively compact in $X$, for every $z\in X$. 
\end{enumerate}
Since $\mathscr{F}$ comprises semicontractions, condition (i) is automatically satisfied. 

Now, because $X$ is \edit{proper}, a subset $C$ of $X$ is relatively compact in $X$ if and only if it is bounded. Therefore $\mathscr{F}$ is relatively compact in $C(X,X)$ if and only if $\mathscr{F}(z)$ is bounded in $X$, for every $z\in X$. 

To conclude the proof it suffices to observe that $\mathscr{F}(z)$ is bounded if and only if $\mathscr{F}(w)$ is bounded, for any two points $z$,~$w\in X$. 
Indeed, fix $z_0\in X$. Then for every $f\in\mathscr{F}$ we have
\[
d\bigl(z_0,f(w)\bigr)\leq d\bigl(z_0,f(z)\bigr)+d\bigl(f(z),f(w)\bigr)\leq d\bigl(z_0,f(z)\bigr)+d(z,w)\;.
\]
So if $\mathscr{F}(z)$ is bounded then $\mathscr{F}(w)$ is bounded. An analogous argument yields the converse.
\end{proof}

\section{Hyperbolic distortion}
\label{section3}

\edit{We denote by $\kappa_X$ the hyperbolic metric on a hyperbolic Riemann surface $X$, and we write $\kappa$ for the hyperbolic metric on $\mathbb{D}$. We refer the reader to  \cite{Ab2022}*{Chapter 1} for the main properties of the hyperbolic metric and distance on hyperbolic Riemann surfaces.}

The hyperbolic metric satisfies a Schwarz--Pick lemma, as follows (see, for example, \cite{Ab2022}*{Theorem 1.9.23}). Here we denote \edit{by $O$ the origin of $T_zX$, the complex tangent space to $X$ at~$z$.}

\begin{theorem}
\label{th:SP}
Let $X$ and $Y$ be two hyperbolic Riemann surfaces and $f\colon X\longrightarrow Y$ a holomorphic map. Then
\[
\kappa_Y\bigl(f(z);df_z(\xi)\bigr)\leq \kappa_X(z;\xi),
\]
for all $z\in X$ and $\xi\in T_zX$. Furthermore, equality holds for some $z\in X$ and $\xi\in T_zX\setminus\{O\}$ if and only if $f$ is a covering map, and if $f$ is a covering map then in fact equality holds everywhere.
\end{theorem}

Following  \cite{BeMi2007}, we can use the hyperbolic metric to measure the distortion of a holomorphic map.

\begin{definition}
\label{def:distortion}
Let $f\in\Hol(X,Y)$ be a holomorphic map between two hyperbolic Riemann surfaces. The \emph{hyperbolic distortion} of~$f$ is the continuous map $f^\#\colon X\longrightarrow\R^+$ given by
\[
f^\#(z)=\frac{\kappa_Y\bigl(f(z);df_z(\xi)\bigr)}{\kappa_X(z;\xi)},
\]
for $z\in X$ and $\xi\in T_zX\setminus\{O\}$. This definition is independent of the choice of $\xi$ because $df_z$ is a complex linear map. 
\end{definition}

\begin{remark}
\label{rem:hypdistD}
The hyperbolic distortion of a map $f\in\Hol(\D,\D)$ is given by
\[
f^\#(z)=|f'(z)|\frac{1-|z|^2}{1-|f(z)|^2}\;.
\]
In particular, $f^\#=|f^h|$, where $f^h$ is the hyperbolic derivative of~$f$ (see \cite{Ab2022}*{Section 1.5}). 
\end{remark}

\edit{The remainder of this section concerns standard properties of hyperbolic distortion that can be found in, for example,  \cite{BeMi2007,GuKoMoRo2025} for holomorphic self-maps of the unit disc -- the more general statements for hyperbolic Riemann surfaces presented here can be proven using lifting arguments. \emph{Or see \href{https://arxiv.org/abs/2504.21124v1}{Version 1} of this manuscript for proofs.}

The first lemma lists some of the basic properties of hyperbolic distortion. Parts (i) to (iii) follow from Theorem~\ref{th:SP}, part (iv) is similar to the case $X=\mathbb{D}$ largely covered in \cite{BeMi2007}, and part (v) follows from classical complex analysis.}

\begin{lemma}
\label{th:chain}
Let $f\in\Hol(X,Y)$ for hyperbolic Riemann surfaces $X$ and $Y$. 
\begin{enumerate}
\item We have $0\leq f^\#(z)\leq 1$ for all $z\in X$.
\item If $f$ is a covering map then $f^\#(z)=1$ for all $z\in X$. 
\item If there exists $z_0\in X$ such that $f^\#(z_0)=1$ then $f$ is a covering map.
\item Let $Z$ be a hyperbolic Riemann surface and $g\in\Hol(Y,Z)$. Then $(g\circ f)^\#=(g^\#\circ f)f^\#$.
\item If $\{f_n\}\subset\Hol(X,Y)$ is a sequence of holomorphic maps that  converges to $f\in\Hol(X,Y)$ then $f_n^\#\to f^\#$ uniformly on compact subsets of $X$.
\end{enumerate}
\end{lemma}

We also have the following elementary lemma\edit{, which can be proven by integrating the hyperbolic metric along suitable paths.}

\begin{lemma}
\label{lemma1}
Let $f\in\Hol(X,Y)$ for hyperbolic Riemann surfaces $X$ and $Y$. \edit{Let $K$ be the closed hyperbolic disc in $X$ centred at $z$ of radius $r$, and} let $\ell_K = \sup_{z\in K} f^\#(z)$. Then 
\[
\omega_Y\bigl(f(z),f(w)\bigr)\leq \ell_K \omega_X(z,w),
\]
for all $z,w\in K$.
\end{lemma}

Using universal covering maps and the hyperbolic distortion of self-maps of~the unit disc \edit{we can compute  hyperbolic distortion in terms of hyperbolic distance. Again we omit the proof; the same result for hyperbolic domains is stated in \cite[Section~11]{BeMi2007}.}

\begin{theorem}
\label{th:hddist}
Let $f\in\Hol(X,Y)$ for hyperbolic Riemann surfaces $X$ and $Y$. Then
\[
f^\#(z)=\lim_{z'\to z}\frac{\omega_Y\bigl(f(z'),f(z)\bigr)}{\omega_X(z',z)},
\]
for all $z\in X$.
\end{theorem}

The next theorem can be found in \edit{\cite[Theorem~11.2]{BeMi2007} when $X$ and $Y$ are hyperbolic domains, and the more general version for hyperbolic Riemann surfaces can be proved in much the same way.} Observe that, if $f$ is not a covering map, then the image of $f^\#$ lies in the unit disc and so we can measure the hyperbolic distance between any two points $f^\#(z)$ and $f^\#(w)$.

\begin{theorem}
\label{theorem2}
Let $f\in\Hol(X,Y)$ for hyperbolic Riemann surfaces $X$ and $Y$. Assume that $f$ is not a covering map. Then 
\begin{equation*}
\omega\bigl(f^\#(z),f^\#(w)\bigr)\leq 2\omega_X(z,w)\;,
\label{eq:th2}
\end{equation*}
for all $z,w\in X$.
\end{theorem}

One can use this estimate to prove a relationship between hyperbolic distortions at two different points.

\begin{corollary}
\label{corollary1}
Let $f\in\Hol(X,Y)$ for hyperbolic Riemann surfaces $X$ and $Y$. Then 
\[
1-f^\#(z) \leq 2e^{4\omega_X(z,w)}\bigl(1-f^\#(w)\bigr)\;,
\]
for all $z,w\in X$.
\end{corollary}

\edit{This corollary is known when $f$ is a self-map of the unit disc (see formula (3.4) of \cite{GuKoMoRo2025}). The more general statement follows quickly from Theorem~\ref{theorem2}.}

An immediately consequence of Corollary~\ref{corollary1} is the following observation\edit{; the proof (omitted) is straightforward.}

\begin{corollary}
\label{th:3.6}
Let $\{f_n\}\subset\Hol(X,Y)$ for hyperbolic Riemann surfaces $X$ and $Y$. Then the following assertions are equivalent.
\begin{enumerate}
\item For all $z\in X$ we have $\sum_n\bigl(1-f_n^\#(z)\bigr)<+\infty$.
\item There exists $z_0\in X$ such that $\sum_n\bigl(1-f_n^\#(z_0)\bigr)<+\infty$.
\item For any sequence $\{z_n\}$ relatively compact in~$X$ we have $\sum_n\bigl(1-f_n^\#(z_n)\bigr)<+\infty$.
\item There exists a sequence $\{z_n^o\}$ relatively compact in~$X$ such that $\sum_n\bigl(1-f_n^\#(z_n^o)\bigr)<+\infty$.
\end{enumerate}
\end{corollary}

\section{Straightening of left iterated function systems}\label{section4}

We begin this section by proving Theorem~\ref{theoremA}.

\begin{proof}[Proof of Theorem~\ref{theoremA}]
We choose $\gamma_n\in\Aut(\D)$ with $\gamma_n(0)=L_n(0)$ and let $H_n=\gamma_n^{-1}\circ L_n$. Then $H_n(0)=0$. For $z\in\D$, we have
\[
\omega(H_n(z),0) = \omega(H_n(z),H_n(0))=\omega(L_n(z),L_n(0)).
\]
Since $\omega(L_n(z),L_n(0))\leq \omega(L_{n-1}(z),L_{n-1}(0))$, it follows that $\omega(H_n(z),0)\leq \omega(H_{n-1}(z),0)$, so $|H_1(z)|\geq |H_2(z)|\geq\dotsb$. One possibility is that $\{H_n\}$ converges to the constant map with value 0. If this is not so, then there exists $w\in \D$ for which $\{|H_n(w)|\}$ converges to a positive constant. Let $\theta_n$ be an argument of $H_n(w)$. By pre-composing $\gamma_n$ with the rotation $e^{i\theta_n}z$, we can assume that $\{H_n(w)\}$ is a sequence of positive numbers, so it converges to a positive number $w_0$.

Observe that the sequence $\{H_n\}$ is relatively compact, by Theorem~\ref{theorem1}. Suppose there are two subsequences $\{H_{m_i}\}$ and $\{H_{n_j}\}$ of $\{H_n\}$ with limits $h$ and $k$ respectively. Each of $h$ and $k$ fixes $0$ and $h(w)=k(w)=w_0$. By passing to further subsequences we can assume that $m_1<n_1<m_2<n_2<\dotsb$. Let $K_i=\gamma_{n_i}^{-1}\circ f_{n_i}\circ \dots\circ f_{m_i+1}\circ \gamma_{m_i}$. Then $H_{n_i}=K_i\circ H_{m_i}$. Note that $K_i(0)=0$, so $\{K_i\}$ is relatively compact. Consequently, there is a subsequence of $\{K_i\}$ with limit $\psi\in \Hol(\mathbb{D},\mathbb{D})$, where $\psi \circ h=k$. Notice that $\psi(0)=0$ and $\psi(w_0)=\psi(h(w))=k(w)=w_0$. It follows that $\psi=\id_\D$ since, among all holomorphic self-maps of $\mathbb{D}$, only the identity map fixes two distinct points. Hence $h=k$ and $H_n\to h$, as required.

It remains to prove that $h$ is unique up to left composition by elements of $\Aut(\D)$. Suppose then that there are sequences $\{\gamma_n\}$ and $\{\delta_n\}$ in $\Aut(\D)$, and $h,k\in\Hol(\D,\D)$, with $\gamma_n^{-1}\circ L_n\to h$ and $\delta_n^{-1}\circ L_n\to k$. Let $\phi_n=\gamma_n^{-1}\circ \delta_n$. Then
\begin{align*}
\omega (\phi_n(k(0)),h(0)) 
&\leq \omega (\phi_n(k(0)),\phi_n(\delta_n^{-1}(L_n(0))))+\omega (\phi_n(\delta_n^{-1}(L_n(0))),h(0))\\
& = \omega(k(0),\delta_n^{-1}(L_n(0)))+\omega(\gamma_n^{-1}(L_n(0)),h(0)).
\end{align*}
Hence $\omega (\phi_n(k(0)),h(0))\to 0$, so $\{\phi_n\}$ is relatively compact. It follows that it has a subsequence converging to $\phi\in \Aut(\D)$. Now, $\gamma_n^{-1}\circ L_n = \phi_n\circ (\delta_n^{-1}\circ L_n)$, so $h=\phi\circ k$, as required. 
\end{proof}

\begin{remark}
\label{rem:h0}
The function $h\in\Hol(\D,\D)$ built in the previous proof is such that $h(0)=0$. 
\end{remark}

\begin{definition}
\label{def:lefts}
Let $\{f_n\}\subset\Hol(\D,\D)$ and $L_n=f_n\circ f_{n-1}\circ \dots\circ f_1$. A function $h\in\Hol(\D,\D)$ is called a \emph{left straightening} of $\{f_n\}$ if $h(0)=0$ and there is a sequence $\{\gamma_n\}$ of conformal automorphisms of~$\D$ with $\gamma_n^{-1}\circ L_n \to h$. In the particular case when $f_n= f$ for all $n\in\mathbb{N}$, we say that $h$ is a left straightening of $f$.
\end{definition}

We saw in the introduction that $L_n^\#\to h^\#$ (Corollary~\ref{corollaryB}) and 
\begin{equation}
\lim_{n\to+\infty}\omega\bigl(L_n(z),L_n(w)\bigr)=\omega\bigl(h(z),h(w)\bigr)\;,
\label{eq:C}
\end{equation}
for $z,w\in\mathbb{D}$ (Corollary~\ref{corollaryC}). In light of Remark~\ref{rem:h0}, we can now assume that $h(0)=0$, so $h$ is a left straightening of $\{f_n\}$.

We also noted in the introduction that Corollary~\ref{corollaryC} has applications to the theory of iteration of a single holomorphic self-map of $\mathbb{D}$. Here we expand on these applications. Given $f\in\Hol(\mathbb{D},\mathbb{D}))$ and $\mu\in\mathbb{N}$, the \emph{hyperbolic $\mu$-step} $s^f_\mu$ of $f$ is defined by the limit
\[
s^f_\mu(z)=\lim_{n\to+\infty}\omega\bigl(f^n(z),f^{n+\mu}(z)\bigr)\;;
\] 
see \cite{Ab2022}*{Section 4.6} and references therein. We can use the left straightening to compute $s^f_\mu$.

\begin{corollary}
\label{th:mustep}
Let $f\in\Hol(\D,\D)$ and $\mu\in\mathbb{N}$. Then for every $z\in\D$ we have
\[
s^f_\mu(z)=\omega\bigl(h(z),h\bigl(f^\mu(z)\bigr)\bigr)\;,
\]
where $h$ is a left straightening of $f$.
\end{corollary}
\begin{proof}
This follows from \eqref{eq:C} applied with $w=f^\mu(z)$.
\end{proof}

We recall that a holomorphic map $f\in\Hol(\D,\D)\setminus\Aut(\D)$ falls in one of the following three classes:
\begin{enumerate}
\item $f$ is \emph{elliptic} if it has a fixed point in~$\D$;
\item $f$ is \emph{parabolic} if it has no fixed points in $\D$ and $f'(\tau_f)=1$;
\item $f$ is \emph{hyperbolic} if it has no fixed points in $\D$ and $0<f'(\tau_f)<1$.
\end{enumerate} 
Here $\tau_f\in\partial\D$ is the Wolff point of~$f$ and $f'(\tau_f)$ is the angular derivative of~$f$ at~$\tau_f$, which necessarily belongs to the interval~$(0,1]$. 

Left straightenings relate closely to this classification. 

\begin{proposition}
\label{th:lsclass}
Let $f\in\Hol(\D,\D)$. 
\begin{enumerate}
\item The map $f$ is an automorphism if and only if a (and hence any) left straightening of $f$ is an automorphism.
\item The map $f$ is either elliptic or parabolic with zero hyperbolic $1$-step if and only if a (and hence any) left straightening of $f$ is constant.
\item The map $f$ is either hyperbolic or parabolic with positive hyperbolic $1$-step if and only if a (and hence any) left straightening of $f$ is not constant and not an automorphism. 
\end{enumerate}
\end{proposition}

\begin{proof}
(i) By Corollary~\ref{corollaryB} we have that $(f^n)^\#(z)\to h^\#(z)$, for $z\in\mathbb{D}$, where $h$ is a left straightening of $f$. Since $(f^n)^\#(z)=f^\#(f^{n-1}(z))(f^{n-1})^\#(z)$, the sequence $\{(f^n)^\#(z)\}$ is decreasing. From Lemma~\ref{th:chain} we see that $f^\#(z)=1$ for any (and hence all) $z\in\mathbb{D}$ if and only if $h^\#(z)=1$ for any (and all) $z\in\mathbb{D}$. Therefore $f\in \Aut(\mathbb{D})$ if and only if $h\in\Aut(\mathbb{D})$.

(ii) Assume now that $f$ is elliptic and not an automorphism. Then $\{f^n\}$ converges to the constant map~$z_0$. By taking $\gamma_n=\id_\D$ for each $n\in\mathbb{N}$, we see that $\{\gamma_n^{-1}\circ f^n\}$ also converges to the constant map~$z_0$. Hence, by the uniqueness statement of Theorem~\ref{theoremA}, any left straightening of $f$ is constant.

Assume, instead, that $f$ is parabolic with zero hyperbolic $1$-step, and let $h$ be a left straightening of~$f$. Then Corollary~\ref{corollaryC} and \cite{Ab2022}*{Corollary 4.6.9.(iv)} yield
\[
\omega\bigl(h(z),h(w)\bigr)=\lim_{n\to+\infty}\omega\bigl(f^n(z),f^n(w)\bigr)=0\;,
\]
for all $z,w\in\D$, so again $h$ is constant.

Conversely, if $h$ is constant then, by Corollary~\ref{th:mustep}, $f$ has zero hyperbolic $1$-step and thus cannot be either hyperbolic (by \cite{Ab2022}*{Corollary 4.6.9.(ii)})
or parabolic with positive hyperbolic $1$-step (by definition). Given (i), we see that $f$ is either elliptic or parabolic with zero hyperbolic $1$-step.

(iii) This follows from (i) and (ii), because there are no other possibilities.
\end{proof}

Corollary~\ref{th:limBenal} (in the next section)  has another characterization of functions with constant left straightening, expressed in terms of the hyperbolic distortion.

\begin{remark}
\label{rem:lsmus}
Corollary~\ref{th:mustep} and the properties of the hyperbolic $\mu$-step yield some interesting relationships between $f$ and any left straightening $h$ when $f$ is hyperbolic or parabolic with positive hyperbolic $1$-step. For instance, since the hyperbolic $1$-step is either identically zero or never vanishing (by \cite{Ab2022}*{Corollary 4.6.9.(i)}), if $f$ is hyperbolic or parabolic with positive hyperbolic $1$-step then $h\bigl(f(z)\bigr)\ne h(z)$ for all $z\in\D$. 

Furthermore, by combining Lemma 4.6.4, Proposition 4.6.6, and Corollary 4.6.9 from \cite{Ab2022} with Corollary~\ref{th:mustep} we get that if $f$ is hyperbolic or parabolic then
\[
\inf_{z\in\D}\omega\bigl(h(z),h\bigl(f(z)\bigr)\bigr)=\lim_{\mu\to +\infty}\frac{\omega\bigl(h(z),h\bigl(f^\mu(z)\bigr)\bigr)}{\mu}=\frac{1}{2}\log\frac{1}{f'(\tau_f)}\;,
\]
where the middle limit is independent of $z\in\D$. 

Finally, \cite{Ab2022}*{Lemma 4.6.4} implies that, for every $z\in\D$, the sequence $\bigl\{\omega\bigl(h(z),h\bigl(f^n(z)\bigr)\bigr)\bigr\}$ is subadditive.
\end{remark}

We recall from the introduction that a map $f\in\Hol(\D,\D)$ is \emph{left semiconjugate} to another map $g\in\Hol(\D,\D)$ if there is a nonconstant map $\phi\in\Hol(\D,\D)$ such that $\phi\circ f=g\circ \phi$. The following corollary (of Theorem~\ref{theoremA}) is a more general version of Corollary~\ref{corollaryD}.

\begin{corollary}
\label{th:D}
Let $f\in\Hol(\D,\D)$. The following statements are equivalent.
\begin{enumerate}
\item The map $f$ is left semiconjugate in $\Hol(\D,\D)$ to an automorphism of $\D$.
\item Any left straightening of $f$ is nonconstant.
\item The map $f$ is either an automorphism or else it is hyperbolic or parabolic with positive hyperbolic 1-step.
\end{enumerate}
\end{corollary}
\begin{proof}
By Theorem~\ref{theoremA} we know that there is a sequence $\{\gamma_n\}$ of automorphisms of $\D$ such that $\gamma_n^{-1}\circ f^n\to h$, where $h$ is a left straightening of~$f$.

First we prove that (ii) implies (i). Suppose that $h$ is nonconstant. Let $g_n=\gamma_n^{-1}\circ f^n$ and  $\phi_n=\gamma_{n}^{-1}\circ \gamma_{n+1}$. Then $g_n \circ f = \phi_n\circ g_{n+1}$. Observe that
\[
\begin{aligned}
\omega\bigl(\phi_n\bigl(h(0)\bigr),h\bigl(f(0)\bigr)\bigr) 
&\leq \omega\bigl(\phi_n\bigl(h(0)\bigr),\phi_n\bigl(g_{n+1}(0)\bigr)\bigr)+\omega(g_n\bigl(f(0)\bigr),h\bigl(f(0)\bigr)\bigr)\\
&=\omega\bigl(h(0),g_{n+1}(0)\bigr)+\omega\bigl(g_n\bigl(f(0)\bigr),h\bigl(f(0)\bigr)\bigr)\;.
\end{aligned}
\]
Since $g_n\to h$, the right-hand side is uniformly bounded. It follows from Theorem~\ref{theorem1} that the sequence $\{\phi_n\}$ is relatively compact, so it admits a subsequence that converges to an automorphism $\phi$ of $\D$. Then $h \circ f=\phi\circ h$, so $f$ is left semiconjugate to $\phi$.

Next we prove that (i) implies (ii). Suppose that $\psi \circ f = \chi \circ \psi$, where $\psi\in\Hol(\D,\D)$ is nonconstant and $\chi\in\Aut(\D)$. Choose $z,w\in\D$ with $\psi(z)\neq \psi(w)$. Observe that
\begin{align*}
\omega\bigl(\gamma_n^{-1}\circ f^n(z),\gamma_n^{-1}\circ f^n(w)\bigr)
&=\omega\bigl(f^n(z),f^n(w)\bigr)\\
&\geq \omega\bigl(\psi\bigl(f^n(z)\bigr),\psi\bigl(f^n(w)\bigr)\bigr)\\
&=\omega\bigl(\chi^n\bigl(\psi(z)\bigr),\chi^n\bigl(\psi(w)\bigr)\bigr)\\
&=\omega\bigl(\psi(z),\psi(w)\bigr)>0
\end{align*}
and $\omega\bigl(\gamma_n^{-1}\circ f^n(z),\gamma_n^{-1}\circ f^n(w)\bigr)\to \omega\bigl(h(z),h(w)\bigr)$. Hence $h(z)\ne h(w)$ and $h$ is not constant, as claimed.

The equivalence of (ii) and (iii) follows from Proposition~\ref{th:lsclass}.
\end{proof}

\section{Necessary and sufficient conditions for nonconstant limits}\label{section5}

In this section we prove Theorem~\ref{theoremE} and some related results.

\begin{proposition}
\label{th:Benal}
Let $\{f_n\}\subset\Hol(\D,\D)$. Then any left straightening of $\{f_n\}$ is constant if and only if
\[
\sum_{n=1}^\infty \bigl(1-f^\#_n\bigl(L_{n-1}(z)\bigr)\bigr)=+\infty
\]
for some (and hence all) $z\in\D$.
\end{proposition}

\begin{proof}
Let $h$ be a left straightening of $\{f_n\}$. Theorem~\ref{theoremA} implies that
\[
\omega\bigl(L_n(z),L_n(w)\bigr)\to\omega\bigl(h(z),h(w)\bigr)
\]
for all $z,w\in\D$. Thus $h$ is constant if and only if
\[
\lim_{n\to +\infty}\omega\bigl(L_n(z),L_n(w)\bigr)=0
\]
for all $z,w\in\D$. The assertion then follows from \cite{BeEvFaRiSt2022}*{Theorem~2.1 and Corollary 2.2} applied to $g_n=\gamma_n\circ f_n\circ\gamma_{n-1}^{-1}$,
where, after fixing a point $z\in\D$, the automorphism $\gamma_n\in\Aut(\D)$ is chosen such that $\gamma_0(z)=0$ and $\gamma_n\bigl(L_n(z)\bigr)=0$, for $n\in\mathbb{N}$. Notice that the hyperbolic distortion used in  \cite{BeEvFaRiSt2022} coincides with ours thanks to Theorem~\ref{th:hddist}.
\end{proof}

\begin{corollary}
\label{th:limBenal}
Let $f\in\Hol(\D,\D)$. Then the following statements are equivalent.
\begin{enumerate}
\item Any left straightening of $f$ is constant.
\item The map $f$ is either elliptic or parabolic with zero hyperbolic $1$-step.
\item We have
\[
\sum_{n=1}^\infty\bigl(1-f^\#\bigl(f^{n-1}(z)\bigr)\bigr)=+\infty
\]
for some (and hence all) $z\in\D$.
\end{enumerate}
\end{corollary}
\begin{proof}
This follows from Propositions~\ref{th:lsclass} and \ref{th:Benal}.
\end{proof}

\edit{Next we prove Theorem~\ref{theoremE}. That theorem asserts the equivalence of four statements (i) to (iv) for a hyperbolic Riemann surface $X$. Our proof shows that, additionally, in the case $X=\mathbb{D}$ these four statements are equivalent to the fifth statement 
\begin{enumerate}
\item[\textnormal{(v)}] \emph{Any left straightening of $\{f_n\}$ is nonconstant.}
\end{enumerate}
}

\begin{proof}[Proof of Theorem~\ref{theoremE}]
Since (i)$\Longrightarrow$(ii) is trivial and (iii)$\Longrightarrow$(iv) follows from Corollary~\ref{th:3.6}, it suffices to prove (ii)$\Longrightarrow$(iii) and (iv)$\Longrightarrow$(i).
Along the way, we shall also show that, when $X=\D$, (ii) implies (v) and (v) implies (iii).

By the chain rule (Lemma~\ref{th:chain}), we have
\[
L^\#_n(z) = \prod_{j=1}^n f_j^\#\bigl(L_{j-1}(z)\bigr),
\]
for any $z\in\D$. Since $L^\#_n(z)=f^\#_n\bigl(L_{n-1}(z)\bigr)L^\#_{n-1}(z)$ and $f^\#_n\leq 1$, the sequence $\{L^\#_n(z)\}$ is monotonic, so it converges to a finite nonnegative number. Consequently, if $F\in\Hol(X,X)$ is a limit point of the sequence $\{L_n\}$, then we have
\begin{equation}
F^\#(z)=\prod_{n=1}^{\infty} f_n^\#\bigl(L_{n-1}(z)\bigr)\;.
\label{eq:infprd}
\end{equation}
Assume that (ii) holds; then we can choose a nonconstant limit point $F$. In particular, there is $z_0\in\D$ such that
$F^\#(z_0)\neq 0$. Then $f_n^\#\bigl(L_{n-1}(z_0)\bigr)\neq 0$ for all $n$, and \eqref{eq:infprd} implies that 
\[
\sum_{n=1}^\infty \bigl(1-f_n^\#\bigl(L_{n-1}(z_0)\bigr)\bigr)<+\infty\;.
\] 
If $X=\D$, then, by Proposition~\ref{th:Benal}, this is equivalent to (v). Furthermore,
since $\{L_{n-1}(z_0)\}$ is relatively compact in $X$, Corollary~\ref{th:3.6} implies that $\sum_n \bigl(1-f_n^\#(z_0)\bigr)<+\infty$, which is (iii).

Assume finally that (iv) holds. Suppose, by contradiction, that $\{L_n\}$ has a constant limit point $F$. Choose $z_1\in X$ and let $z_n=L_{n-1}(z_1)$ for $n>1$. 
Since $F^\#(z_1)=0$, \eqref{eq:infprd} implies that
\[
\sum_{n=1}^\infty \bigl(1-f_n^\#(z_n)\bigr)=+\infty\;.
\] 
But $\{z_n\}$ is relatively compact in~$X$; therefore Corollary~\ref{th:3.6} implies that $\sum_n \bigl(1-f_n^\#(z_1)\bigr)=+\infty$, which gives statement (iv).
\end{proof}

In Section~\ref{section8} we shall give an example of a left iterated function system that is neither relatively compact nor compactly divergent.

\section{Hyperbolic distortion inequality}\label{section6}

In this section we prove Theorem~\ref{theoremF}. To prove this theorem and the next lemma we use the following formulas for the hyperbolic metric (see, for example, \cite{Ab2022}*{Proposition 1.3.10}):
\[
\begin{aligned}
\sinh \omega(z,w) &= \frac{|z-w|}{\sqrt{(1-|z|^2)(1-|w|^2)}}\;,\\
 \cosh \omega(z,w) &= \frac{|1-z\overline{w}|}{\sqrt{(1-|z|^2)(1-|w|^2)}}\;. 
 \end{aligned}
\]

\begin{lemma}
\label{lemmaB}
Let $z$,~$w\in\mathbb{D}$. Then 
\[
|z-w| \leq 2(1-|w|)\sinh\bigl(2\omega(z,w)\bigr)\;.
\]
\end{lemma}
\begin{proof}
Observe that 
\[
\sinh \bigl(2\omega(z,w)\bigr)= 2\sinh\omega (z,w)\cosh\omega(z,w)
= \frac{2|z-w||1-z\overline{w}|}{(1-|z|^2)(1-|w|^2)}\;.
\]
Hence
\[
\sinh \bigl(2\omega(z,w)\bigr)\geq \frac{|z-w|(1-|z||w|)}{2(1-|z|)(1-|w|)}
\geq \frac{|z-w|}{2(1-|w|)}\;.
\]
The result follows on rearranging this inequality.
\end{proof}

We can now prove Theorem~\ref{theoremF}.

\begin{proof}[Proof of Theorem~\ref{theoremF}]
If $f\in\Aut(\D)$ we can choose $\gamma=f$, and the inequality is satisfied. Let us assume, then, that $f\notin\Aut(\D)$.

Suppose first that $w=0$ and $f(0)=0$. Assume also for the moment that $f^\#(0)=0$. Then we can choose $\gamma=\id_\D$, because
\[
\omega\bigl(f(z),z\bigr)\leq \omega\bigl(f(z),0\bigr)+\omega(z,0)\leq 2\omega(z,0) < 2e^{4\omega(z,0)}\;.
\]
Assume now that $f^\#(0)\neq 0$ (but still $w=0$ and $f(0)=0$). Then there is $g\in\Hol(\D,\D)$ such that $f(z)=zg(z)$ for all $z\in\D$; moreover, $|g(0)|=f^\#(0)\ne 0$. Let $\alpha=g(0)/|g(0)|$ and let $\gamma(z)=\alpha z$. Since $|z|\leq 1$ and $|g(z)|\leq 1$, we have
\[
\begin{aligned}
\sinh \omega\bigl(f(z),\gamma(z)\bigr)
&= \frac{|z||g(z)-\alpha|}{\sqrt{(1-|zg(z)|^2)(1-|z|^2)}}\\
&\leq \frac{1}{1-|z|^2}\bigl(|g(z)-g(0)|+|g(0)-\alpha|\bigr)\;.
\end{aligned}
\] 
Observe that $|g(0)-\alpha|=1-|g(0)|$. From Lemma~\ref{lemmaB} we have
\[
|g(z)-g(0)|\leq 2(1-|g(0)|)\sinh \bigl(2\omega\bigl(g(z),g(0)\bigr)\bigr)\leq 2(1-|g(0)|)\sinh \bigl(2\omega(z,0)\bigr)\;.
\]
Hence
\[
\begin{aligned}
\sinh \omega\bigl(f(z),\gamma(z)\bigr)
&\leq \frac{1}{1-|z|^2}\mleft(2\sinh\bigl(2\omega(z,0)\bigr)+1\mright)(1-|g(0)|)\\
&\leq e^{2\omega(0,z)}\mleft(2\sinh\bigl(2\omega(z,0)\bigr)+1\mright)(1-|g(0)|)\\
&\leq 2e^{4\omega(0,z)}(1-|g(0)|)\;.
\end{aligned} 
\]
Since $|g(0)|=f^\#(0)$, the result is now established in this special case.

Consider now any map $f\in\Hol(\D,\D)$ that is not an automorphism and any point $w\in\D$. Choose $\phi$,~$\psi\in \Aut(\D)$ with $\phi(0)=w$ and $\psi(0)=f(w)$. Let $g=\psi^{-1}\circ f\circ \phi$. Then $g(0)=0$ and the preceding argument tells us that we can find $\tilde\gamma\in\Aut(\D)$ with 
\[
\omega\bigl(g(z),\tilde\gamma(z)\bigr)\leq 2e^{4\omega(z,0)}\bigl(1-g^\#(0)\bigr)\;,
\]
for all $z\in\D$. Let $\gamma=\psi\circ\tilde\gamma\circ \phi^{-1} \in\Aut(\D)$. With $\zeta=\phi(z)$, we have 
\[
\omega\bigl(g(z),\tilde\gamma(z)\bigr)=\omega\bigl(f(\zeta),\gamma(\zeta)\bigr)\;.
\] 
Also, $\omega(z,0)=\omega(\zeta,w)$ and $g^\#(0)=f^\#(w)$. Hence
\[
\omega\bigl(f(\zeta),\gamma(\zeta))\leq 2e^{4\omega(\zeta,w)}\bigl(1-f^\#(w)\bigr)\;,
\]
for all $\zeta\in\D$, as required.
\end{proof}

\edit{
\begin{remark}
The referee observed that the automorphism $\gamma$ constructed in the proof of Theorem~\ref{theoremF} is the unique automorphism with $\gamma(w)=f(w)$ and for which the argument of $\gamma'(w)$ coincides with that of $f'(w)$. It can be shown by perturbing $\gamma$ that some other automorphisms also satisfy the inequality of Theorem~\ref{theoremF}; indeed, by optimising the proof, that inequality could be improved slightly (for example, the coefficient 2 could be reduced to $\tfrac54$).
\end{remark}
}

\begin{remark}
\label{rem:add}
Theorem~\ref{theoremF} implies that if $\{f_n\}\subset\Hol(\D,\D)$ satisfies $\sum_n\bigl(1-f_n^\#(w)\bigr)<+\infty$ for some point $w$ in $\D$, then there is a sequence $\{\gamma_n\}$ in $\Hol(\D,\D)$ with $\sum_n \omega\bigl(f_n(z),\gamma_n(z)\bigr) <+\infty$ for all $z\in\D$. Let $\Gamma_n=\gamma_n\circ \gamma_{n-1}\circ \dots\circ \gamma_1$; then
\[
\omega\bigl(\Gamma_{n-1}^{-1}\circ L_{n-1}(z),\Gamma_n^{-1}\circ L_n(z)\bigr)=\omega\bigl(\gamma_n\bigl(L_{n-1}(z)\bigr),f_n\bigl(L_{n-1}(z)\bigr)\bigr)\;.
\]
If (as in Theorem~\ref{theoremE}) $\{L_n\}$ is relatively compact in  $\Hol(\D,\D)$, then by Corollary~\ref{th:3.6} we obtain
\[
\sum_{n=1}^\infty \omega\bigl(\Gamma_{n-1}^{-1}\circ L_{n-1}(z),\Gamma_n^{-1}\circ L_n(z)\bigr)<+\infty\;,
\] 
and, consequently, there exists $h\in \Hol(\D,\D)$ with $\Gamma_n^{-1}\circ L_n\to h$. In this way we have recovered the outcome of Theorem~\ref{theoremA}.
\end{remark}

\edit{We finish this section by proving Corollary~\ref{corollaryG} as a consequence of Theorem~\ref{theoremF}. }

\begin{proof}[Proof of Corollary~\ref{corollaryG}]
Applying Theorem~\ref{theoremF} with $w=z_n$ we can find $\gamma_n\in\Aut(\D)$ with 
\[
\omega\bigl(f(z),\gamma_n(z)\bigr)\leq 2e^{4\omega(z,z_n)}\bigl(1-f^\#(z_n)\bigr)\;,
\]
for all $z\in\D$. From the inequality $\omega(z,z_n)\leq \omega(z,0)+\omega(0,z_n)$, we see that 
\[
\omega\bigl(f(z),\gamma_n(z)\bigr)\leq \frac{32}{(1-|z|)^2}\frac{1-f^\#(z_n)}{(1-|z_n|)^2}\;.
\]
The hypothesis of the corollary then tells us that $\gamma_n\to f$ and so $f\in \Aut(\D)$, as claimed.
\end{proof}

\section{Constant limits of left iterated function systems}\label{section7}

In this section we prove Theorems~\ref{theoremH} and~\ref{theoremI}. The next lemma is a particular case of the continuous dependence of the Wolff point on the corresponding map (see \cite{Heins1941} and \cite{Ab2022}*{Theorem~3.4.2}). 

\begin{lemma}
\label{lemma1b}
Let $f$ be a map in $\Hol(X,X)\setminus\{\id_X\}$, for a hyperbolic Riemann surface $X$, with a fixed point $a_0\in X$. Suppose that $\{f_n\}$ is a sequence in $\Hol(X,X)$ that converges to $f$. Then there exists $N\in\mathbb{N}$ such that if $n\ge N$ then each $f_n$ has a fixed point $a_n\in X$. Furthermore, we can choose each $a_n$ such that $a_n\to a_0$ as $n\to+\infty$.
\end{lemma}

\begin{remark}
\label{rem:fixC}
The first assertion in Lemma~\ref{lemma1b} holds for the Riemann sphere, for the trivial reason that every holomorphic self-map of the Riemann sphere has a fixed point. It also holds when $X=\C$. To see this, assume that $\{f_n\}\subset\Hol(\C,\C)$ is a sequence of maps without fixed points that converges to a map $f\in\Hol(\C,\C)\setminus\{\id_\C\}$ with a fixed point. Since $f_n$ has no fixed points, the function $g_n=f_n-\id_\C$ has no zeros; moreover, $g_n\to f-\id_\C$. Since $f\neq\id_\C$, Hurwitz's theorem implies that $f$ cannot have fixed points, which is a contradiction.
\end{remark}

\begin{remark}
\label{rem:fp}
Let $X$ be a hyperbolic Riemann surface and suppose there exists $f\in\Hol(X,X)$ with two distinct fixed points; then $f=\id_X$ or $X$ is multiply connected and $f$ is a periodic automorphism of $X$ (see \cite{Ab2022}*{Corollary 3.1.16}).
\end{remark}

\begin{definition}
\label{def:hol0}
Let $X$ be a hyperbolic Riemann surface. We denote by $\Hol_0(X,X)$ the subset of $\Hol(X,X)\setminus\{\id_X\}$ of self-maps that have a fixed point in $X$. By Lemma~\ref{lemma1b}, $\Hol_0(X,X)$ is an open subset of $\Hol(X,X)$.
\end{definition}

We can now prove Theorem~\ref{theoremH}.

\begin{proof}[Proof of Theorem~\ref{theoremH}]
Observe that 
\[
\omega_X\bigl(f_n(a),a\bigr)\leq \omega_X\bigl(f_n(a),L_n(a)\bigr)+\omega_X\bigl(L_n(a),a\bigr)\leq \omega_X\bigl(a,L_{n-1}(a)\bigr)+\omega_X\bigl(L_n(a),a\bigr)\;.
\]
Hence $f_n(a)\to a$. 

We claim that $f_n\in\Hol_0(X,X)$ for all $n$ large enough. If not, then we can find a subsequence $\{f_{r_i}\}$ disjoint from $\Hol_0(X,X)$. 
Since $\{f_{r_i}\}$ is relatively compact in $\Hol(X,X)$ (because $f_n(a)\to a$), it has a subsequence $\{f_{s_j}\}$ that converges in $\Hol(X,X)$ to a map $f$ with $f(a)=a$. The hypothesis on $\mathscr{F}$ ensures that $f\neq\id_X$. Hence $f\in \Hol_0(X,X)$ and then, by Lemma~\ref{lemma1b}, $f_{s_j}\in\Hol_0(X,X)$ for $j$ large enough, which is a contradiction.

Thus we can find a positive integer $N$ for which $f_n$ has a fixed point $a_n\in X$, for $n\geq N$. In case $f_n$ has more than one fixed point, we define $a_n$ to be any one of the fixed points of $f_n$ that is closest to $a$.  Suppose that $a_n\nrightarrow a$. Then there is a subsequence $\{f_{r_i}\}$ of $\{f_n\}$ and a positive number $\varepsilon$ with $\omega_X(a_{r_i},a)>\varepsilon$, for all  $i\in\mathbb{N}$. Once again we choose a subsequence $\{f_{s_j}\}$ of $\{f_{r_i}\}$ that converges in $\Hol(X,X)$ to a map $f$; again, $f(a)=a$. By Lemma~\ref{lemma1b}, for sufficiently large values of $j$, the map $f_{s_j}$ has a fixed point $a_{s_j}$ and the resulting sequence of fixed points converges to $a$. This is impossible, because no fixed point of $f_{r_i}$ lies within a distance $\varepsilon$ of $a$. Hence, contrary to our assumption, we have $a_n\to a$, as required.
\end{proof}

\begin{remark}
If $X$ is not a disc, punctured disc, or annulus, then (see, for example, \cite{Ab2022}*{Theorem~2.6.2}) $\id_X$ is isolated in $\Hol(X,X)$, so the hypothesis $\id_X\notin \overline{\mathscr{F}}$ can be omitted.
\end{remark}

Let us now prove Theorem~\ref{theoremI}.

\begin{proof}[Proof of Theorem~\ref{theoremI}]
Let $z_0\in X$ and let $K$ be a compact disc centred at $z_0$ that contains the orbit $\{L_n(z_0)\}$ and all points $\{a_n\}$. We define $\ell_n=\sup_{z\in K} f_n^\#(z)$ and choose $q_n\in K$ such that $\ell_n=f_n^\#(q_n)$. 

Suppose that $\sup_n \ell_n =1$. Then we can find a subsequence $\{n_i\}$ with $q_{n_i}\to q$, $f_{n_i}^\#(q_{n_i})\to 1$, and $f_{n_i}\to f$, where $f\in \Hol(X,X)$. By Lemma~\ref{th:chain}, we have $f^\#(q)=1$, so $f$ is a self-covering of $X$; moreover, $f(a)=a$. It follows from \cite{Ab2022}*{Corollary~3.1.15} that $f$ is a periodic or pseudoperiodic automorphism of $X$, which contradicts one of the hypotheses. 

Hence there exists $\ell\in (0,1)$ with $\ell_n\leq \ell$, for all $n$. Consequently, Lemma~\ref{lemma1} yields
\[
\omega_X\bigl(f_n(z),f_n(w)\bigr)\leq \ell \omega_X(z,w)\;,
\]
for all $z,w\in K$ and $n\in\mathbb{N}$. 

Now let $\varepsilon>0$. Choose $\delta>0$ for which $\ell(1+\delta)<1$ and $N\in\mathbb{N}$ for which $\omega_X(a_n,a_{n-1})<\varepsilon\delta$, for $n\geq N$.  Then
for $n\ge N$ we have
\[
\omega_X\bigl(L_n(z_0),f_n(a_n)\bigr)\leq \ell\omega_X\bigl(L_{n-1}(z_0),a_n)< \ell\left(\omega_X\bigl(L_{n-1}(z_0),a_{n-1}\bigr)+\varepsilon\delta\right)\;. 
\] 
Let $s_n=\omega_X\bigl(L_n(z_0),a_n\bigr)$, for $n\in\mathbb{N}$. Since $a_n$ is a fixed point of $f_n$, we have
\[
s_n< \ell(s_{n-1}+\varepsilon\delta)\;,
\]
for all $n\ge N$.
Consider any integer $m> N$. If $s_{m-1}<\varepsilon$, then $s_m\leq \ell(1+\delta)\varepsilon<\varepsilon$. Alternatively, if $s_{m-1}\geq \varepsilon$, then $s_m\leq \ell(1+\delta)s_{m-1}$. Consequently, $s_n<\varepsilon$ for sufficiently large $n$. 

It follows that $\omega_X\bigl(L_n(z_0),a_n)\to 0$, so $L_n(z_0)\to a$. Since $z_0\in X$ is arbitrary, we see that $\{L_n\}$ converges pointwise (and, by Vitali's theorem, uniformly on compact subsets) to $a$.
\end{proof}

\section{Examples of diverging left iterated function system}
\label{section8}

Here we provide the example promised in the introduction of a sequence $\{f_n\}$ in $\Hol(X,X)$ that converges slowly to $F\in \Hol(X,X)$ for which $\{F^n\}$ is compactly divergent whereas $\{L_n\}$ neither converges in $\Hol(X,X)$ and nor is it compactly divergent.

We choose $X$ to be the upper half-plane $\mathbb{H}^+=\{z\in\C\mid \mathrm{Im}\, z>0\}$ and let $F\in\Hol(\mathbb{H}^+,\mathbb{H}^+)$ be the map $F(z)=z-1$. Then the sequence $\{F^n\}$ is compactly divergent: indeed, it diverges to~$\infty$ on the Riemann sphere.

For $n\in\mathbb{N}$, let $F_n\in\Hol(\mathbb{H}^+,\mathbb{H}^+)$ be given by
\[
F_n(z)=F(z)+\frac{1}{n}i=z-1+\frac{1}{n}i\;.
\] 
Clearly, $F_n\to F$. Let
\[
\varphi_n(z)=\frac{n z-1}{z+n}\;.
\]
It is easy to check that $\varphi_n$ is an elliptic automorphism of~$\mathbb{H}^+$ that fixes~$i$ and $\varphi_n\to\id_{\mathbb{H}^+}$. 

Now let $g_n=\varphi_n\circ F\circ\varphi^{-1}_n$. By construction, each $g_n$ is a parabolic automorphism of~$\mathbb{H}^+$, and since $\varphi_n(\infty)=n$ it follows that $g_n(n)=n$. Consequently, for each $z\in\mathbb{H}^+$, we have $g_n^k(z)\to n$ as $k\to \infty$. Moreover, $g_n\to F$ as $n\to+\infty$.

We shall now build recursively a sequence $\{f_n\}$ in $\Hol(\mathbb{H}^+,\mathbb{H}^+)$ and a sequence $\{m_n\}$ of strictly increasing positive integers, for $n=0,1,2,\dotsc$, such that $\{f_n\}$, $\{m_n\}$, and the left iterated function system $L_n=f_n\circ f_{n-1}\circ \dots \circ f_0$ have the following properties (for $n\in\mathbb{N}$):
\begin{enumerate}
\item[(a)] $m_0=0$ and $m_1=1$;
\item[(b)] $f_0=g_0$ and $f_1=\id_{\mathbb{H}^+}$;
\item[(c)] $m_{2n+1}=m_{2n}+n$;
\item[(d)] $f_{m_{2n-1}+1}=f_{m_{2n-1}+2}=\dots=f_{m_{2n}}=g_n$ and $f_{m_{2n}+1}=f_{m_{2n}+2}=\dots=f_{m_{2n+1}}=F_n$;
\item[(e)] $|L_{m_{2n}}(i)|>n-1/2^n$ and $|L_{m_{2n+1}}(i)-i|<1/2^n$.
\end{enumerate}

First we establish properties (c)--(e) for $n=1$. Choose a positive integer $k_1$ such that $|g_1^{k_1}\bigl(L_{m_1}(i)\bigr)-1|<1/2$. Then
\[
|g_1^{k_1}\bigl(L_{m_1}(i)\bigr)|>1-\frac{1}{2}\;.
\]
Moreover, we have
\[
F_1\bigl(g_1^{k_1}\bigl(L_{m_1}(i)\bigr)\bigr)-i=g_1^{k_1}\bigl(L_{m_1}(i)\bigr)-1\;.
\]
By defining $m_2=m_1+k_1$, $m_3=m_2+1$, 
$f_{m_1+1}=f_{m_1+2}=\dots=f_{m_2}=g_1$, and $f_{m_3}=F_1$ we see that conditions (c)--(e) are satisfied for $n=1$.

Suppose now that we have found integers $m_0<m_1<\cdots<m_{2n-1}$ as well as functions $f_0,f_1,\dots,f_{m_{2n-1}}\in\Hol(\mathbb{H}^+,\mathbb{H}^+)$ satisfying (a)--(e). Choose $k_n\in\mathbb{N}$ such that
\[
|g_n^{k_n}\bigl(L_{m_{2n-1}}(i)\bigr)-n|<\frac{1}{2^n}.
\] 
Then
\[
\bigl|g_n^{k_n}\bigl(L_{m_{2n-1}}(i)\bigr)\bigr|>n-\frac{1}{2^n}\;.
\]
Morever, since $F_n^n(w)=w-n+i$, we have
\[
F_n^n\bigl(g_n^{k_n}\bigl(L_{m_{2n-1}}(i)\bigr)\bigr)-i=g_n^{k_n}\bigl(L_{m_{2n-1}}(i)\bigr)-n\;.
\]
Defining $m_{2n}=m_{2n-1}+k_n$ and $m_{2n+1}=m_{2n}+n$, and choosing $f_{m_{2n-1}+1},f_{m_{2n-1}+2},\dots,f_{m_{2n+1}}$ as in condition (d), we see that conditions (c)--(e) are satisfied for $n$, as required.

In this way we have constructed a sequence $\{f_n\}\subset\Hol(\mathbb{H}^+,\mathbb{H}^+)$ that converges slowly to~$F$ and that generates a left iterated function system $\{L_n\}$ with $L_{m_{2n}}(i)\to \infty$
and $L_{m_{2n+1}}(i)\to i$ as $n\to+\infty$. In particular, $\{L_n\}$ neither converges and nor is it compactly divergent. 

With a less explicit argument we can build another example of a badly behaved left iterated function system.

\begin{proposition}
\label{th:dense}
There exists a sequence $\{\gamma_n\}\subset\Aut(\D)$ with $\gamma_n\to\id_\D$ such that the left iterated function system $\{L_n\}$ generated by~$\{\gamma_n\}$ is dense 
in~$\Aut(\D)$.
\end{proposition}

\begin{proof}
Choose a sequence $\{\mathcal{U}_n\}$ of open neighbourhoods of~$\id_\D$ satisfying the following properties:
\begin{enumerate}
\item $\mathcal{U}_n^{-1}=\mathcal{U}_n$;
\item $\mathcal{U}_{n+1}\subset\mathcal{U}_n$;
\item $\bigcap_n\mathcal{U}_n=\{\id_\D\}$.
\end{enumerate}
For $n\in\mathbb{N}$, the semigroup generated by $\mathcal{U}_n$ is an open subgroup of~$\Aut(\D)$, thanks to property~(i). However, open subgroups of a topological group are also closed. Since $\Aut(\D)$ is connected, it follows that the semigroup generated by $\mathcal{U}_n$ coincides with $\Aut(\D)$. 

Now choose a countable family $\{\phi_j\}\subset\Aut(\D)$ that is dense in~$\Aut(\D)$. Since the semigroup generated by $\mathcal{U}_1$ is~$\Aut(\D)$, we can find
$\gamma_1,\ldots,\gamma_{n_1}\in\mathcal{U}_1$ such that $L_{n_1}:=\gamma_{n_1}\circ\cdots\circ\gamma_1=\phi_1$.  Similarly, since the semigroup generated
by $\mathcal{U}_2$ is~$\Aut(\D)$, we can find $\gamma_{n_1+1},\ldots,\gamma_{n_2}\in\mathcal{U}_2$ such that $\gamma_{n_2}\circ\cdots\circ\gamma_{n_1+1}=\phi_2\circ L_{n_1}^{-1}$; in particular, $L_{n_2}:=\gamma_{n_2}\circ\cdots\circ\gamma_{n_1+1}\circ L_{n_1}=\phi_2$.

Arguing in this way we can find an increasing sequence of natural numbers $\{n_j\}$ and a sequence $\{\gamma_n\}$ of automorphisms such that for each $j\ge 1$ we have
$\gamma_n\in\mathcal{U}_j$ when $n_{j-1}<n\le n_j$ (where $n_0=0$) and $L_{n_j}:=\gamma_{n_j}\circ\cdots\circ\gamma_1=\phi_j$. Then $\{\gamma_n\}$ is as required, because property (iii) implies that $\gamma_n\to\id_\D$. 
\end{proof}


\section{Straightening of right iterated function systems}
\label{section9}

In this section we prove Theorem~\ref{theoremJ}, which is a counterpart to Theorem~\ref{theoremA} for right rather than left iterated function systems. 

\begin{definition}
\label{def:ibo}
Given $\{f_n\}\subset\Hol(X,X)$, a \emph{backward orbit} for the right iterated function system $\{R_n\}$ generated by $\{f_n\}$ is a sequence $\{w_n\}\subset X$ such that
$f_n(w_n)=w_{n-1}$, for all $n\in\mathbb{N}$. In particular, $R_n(w_n)=w_0$ for all $n\in\mathbb{N}$. 
\end{definition}

\begin{proof}[Proof of Theorem~\ref{theoremJ}]
Up to conjugation by an automorphism of~$\D$ we can assume that $w_0=0$. Choose $\gamma_n\in\Aut(\D)$ such that $\gamma_n(w_n)=0$ and 
\[
\frac{\gamma_{n-1}'(w_{n-1})f_{n}'(w_{n})}{\gamma_{n}'(w_{n})}\geq 0\;,
\]
for $n\in\mathbb{N}$, where $\gamma_0=\id_\D$. Let $H_n=R_n\circ \gamma_n^{-1}$ and $g_n = \gamma_{n-1}\circ f_n\circ \gamma_n^{-1}$. Then $\{H_n\}$ is the right iterated function system generated by $\{g_n\}$; moreover, $H_n(0)=0$ and $g_n(0)=0$ for all $n\in\mathbb{N}$. Taking derivatives at 0, we obtain
\[
g_n'(0)=\gamma_{n-1}'\bigl(f_n\bigl(\gamma_n^{-1}(0)\bigr)\bigr)f_n'\bigl(\gamma_n^{-1}(0)\bigr)(\gamma_n^{-1})'(0)=\frac{\gamma_{n-1}'(w_{n-1})f_{n}'(w_{n})}{\gamma_{n}'(w_{n})}\geq 0\;.
\]
Given any compact disc $K$ centred at $0$, we have $g_n(K)\subseteq K$, so $H_n(K)\subseteq H_{n-1}(K)$. Consequently, if a subsequence of $\{H_n\}$ converges to a constant, then this constant must be 0 and the whole sequence must converge to~0.

The other possibility is that no subsequence of $\{H_n\}$ converges to a constant function; we claim that then $\{H_n\}$ converges to a function $h\in\Hol(\D,\D)$. 
Assume, by contradiction, that this is not the case. 
Since $\{H_n\}$ is relatively compact in $\Hol(\D,\D)$ (by Theorem~\ref{theorem1}), there are two convergent subsequences $H_{m_i}\to \phi$ and $H_{n_j}\to \psi$, with $\phi\neq\psi$. Without loss of generality, we can assume that $n_1<m_1<n_2<m_2<\dotsb$; then we can write
\(
H_{m_i}=H_{n_i}\circ \tau_i, 
\)
where $\tau_i=g_{n_i+1}\circ g_{n_i+2}\circ \dotsb\circ g_{m_i}$.
Each member of the sequence $\{\tau_i\}$ fixes $0$, so this sequence is relatively compact and hence has a convergent subsequence with limit $\alpha\in\Hol(\D,\D)$, say. Then $\phi = \psi \circ \alpha$. In a similar manner we obtain $\psi = \phi \circ \beta$, for some $\beta\in \Hol(\D,\D)$. Hence $\phi= \phi\circ\sigma$, where $\sigma=\beta\circ \alpha$; notice that $\sigma(0)=0$. 

Iterating gives $\phi(z)=\phi\bigl(\sigma^n(z)\bigr)$, for all $n\in\mathbb{N}$. If $\sigma^n\to 0$ then $\phi$ is the constant function with value 0, contrary to the assumption that no subsequence of $\{H_n\}$ converges to a constant. Therefore $\sigma$ must be a rotation about the origin. This implies that $\alpha$ and $\beta$ are both rotations about the origin. Now, we know that $g_n'(0)\geq 0$ for all $n$, so $\tau_i'(0)\geq 0$ for all $i$, and hence $\alpha'(0)\geq 0$. Therefore $\alpha$ is the identity map and, thus, $\phi=\psi$, which is a contradiction. It follows that the sequence $\{H_n\}$ must converge to a map $h\in\Hol(\D,\D)$, as claimed.

For uniqueness, suppose that there are two sequences $\{\gamma_n\}$ and $\{\chi_n\}$ in $\Aut(\D)$ satisfying $\gamma_n(w_n)=\chi_n(w_n)=0$ and with $R_n\circ \gamma_n^{-1}\to h$ and $R_n\circ \chi_n^{-1}\to g$, for  $h,g\in \Hol(\D,\D)$. Let $\alpha_n=\gamma_n\circ \chi_n^{-1}$, so $R_n\circ \chi_n^{-1} = (R_n\circ \gamma_n^{-1})\circ \alpha_n$. The automorphism $\alpha_n$ fixes $0$, so we can find a subsequence $\{\alpha_{n_j}\}$ that converges to $\alpha\in\Aut(\D)$. This limit satisfies $g=h\circ\alpha$, as required.
\end{proof}

\section{Convergent right iterated function systems}\label{section10}

Here we prove Theorem~\ref{theoremK}. 

\begin{proof}[Proof of Theorem~\ref{theoremK}] 
Let $\mathscr{S}$ be a relatively compact semigroup in $\Hol(X,X)$ containing $\{R_n\}$. The implication (i)$\implies$(ii) is immediate, and  (iii)$\implies$(iv) follows from Corollary~\ref{th:3.6}. We will now prove that (iv)$\implies$(i); in fact, we prove the contrapositive assertion. 

Suppose that $\{R_n\}$ does not converge to a constant. Fix $z_0\in X$ and let $K=\overline{\mathscr{S}(z_0)}$; this set is compact and $\mathscr{S}$-invariant. Hence $K\supseteq R_1(K)\supseteq R_2(K)\supseteq\dotsb$. Consequently, there exists a subsequence $\{n_i\}$ with $R_{n_i}\to F$, where $F$ is nonconstant. Choose $z\in X$ for which $F^\#(z)\neq 0$. We have $R_{n_i}^\#(z)\to F^\#(z)$. Let $\lambda = F^\#(z)/2$. Then $R_{n_i}^\#(z)>\lambda$, for sufficiently large $n_i$. Now, 
\[
R^\#_{n_i}(z) = \prod_{j=1}^{n_i} f_j^\#\bigl(R_{j,n_i}(z)\bigr)\;, 
\]
where $R_{j,n}=f_{j+1}\circ f_{j+2}\circ\dots\circ f_n$ and $R_{n,n}=\id_X$. Let $\alpha_j=f_j^\#\bigl(R_{j,n_i}(z)\bigr)$. Then 
\[
-\log \prod_{j=1}^{n_i}\alpha_j  = -\sum_{j=1}^{n_i}\log \alpha_j = -\sum_{j=1}^{n_i}\log(1-(1-\alpha_j))\geq \sum_{j=1}^{n_i}(1-\alpha_j)\;,
\]
since $-\log(1-x)\geq x$, for $0\leq x<1$. Let $\mu=-\log \lambda>0$. Then $-\log R_{n_i}^\#(z) <\mu$, so
\[
\sum_{j=1}^{n_i}\bigl(1-f_j^\#\bigl(R_{j,n_i}(z)\bigr)\bigr)<\mu\;,
\]
for all sufficiently large $n_i$. By relative compactness of $\mathscr{S}$ and Corollary~\ref{th:3.6} we deduce that $\sum_n \bigl(1-f_n^\#(z)\bigr)<+\infty$, as required.

It remains to prove that (ii)$\implies$(iii). Once again, we prove the contrapositive assertion. 

Suppose then that $\sum_n\bigl(1-f_n^\#(z_0)\bigr)<+\infty$ for some $z_0\in X$. Let $K=\overline{\mathscr{S}(z_0)}$. Observe that $f_n^\#(z_0)\to 1$, so, by Corollary~\ref{corollary1}, there is a positive integer $N$ for which $f_n^\#(z)\neq 0$, for $n>N$ and $z\in K$. By discarding the first $N$ maps $f_1,f_2,\dots,f_N$ (and then relabelling), we can in fact assume that $f_n^\#(z)\neq 0$ for all $n\in\mathbb{N}$ and $z\in K$.

Next, let $w_n\in K$ be such that $f_n^\#(w_n)=\inf_{z\in K} f_n^\#(z)$. Then $\sum_n \bigl(1-f_n^\#(w_n)\bigr)<+\infty$, by Corollary~\ref{th:3.6}. Hence $\prod_n f_n^\#(w_n)\neq 0$. Now, for $z\in K$,
\[
R^\#_n(z) = \prod_{j=1}^n f_j^\#\bigl(R_{j,n}(z)\bigr)\geq \prod_{j=1}^\infty f_j^\#(w_j)>0. 
\]
Let $F$ be a limit function of $\{R_n\}$. Then $F^\#(z)\neq 0$, for $z\in K$, so $F$ is not constant. Hence $\{R_n\}$ does not contain a subsequence that converges to a constant, as required.
\end{proof}


\begin{bibdiv}
\begin{biblist}
\setlength{\parskip}{0.2pt}
\bib{Ab2022}{book}{
  title={Holomorphic Dynamics on Hyperbolic Riemann Surfaces},
  author={Abate, M.},
  year={2022},
  publisher={De Gruyter}
}

\bib{AbCh2022}{article}{
   author={Abate, M.},
   author={Christodoulou, A.},
   title={Random iteration on hyperbolic Riemann surfaces},
   journal={Ann. Mat. Pura Appl. (4)},
   volume={201},
   date={2022},
   number={4},
   pages={2021--2035},
}

\bib{Ba1988}{book}{
   author={Barnsley, Michael},
   title={Fractals everywhere},
   publisher={Academic Press, Inc., Boston, MA},
   date={1988},
   pages={xii+396},
}

\bib{BaDe1985}{article}{
   author={Barnsley, M. F.},
   author={Demko, S.},
   title={Iterated function systems and the global construction of fractals},
   journal={Proc. Roy. Soc. London Ser. A},
   volume={399},
   date={1985},
   number={1817},
   pages={243--275},
}

\bib{Be2001}{article}{
   author={Beardon, A. F.},
   title={Semi-groups of analytic maps},
   journal={Comput. Methods Funct. Theory},
   volume={1},
   date={2001},
   number={1},
   pages={249--258},
}



\bib{BeCaMiNg2004}{article}{
   author={Beardon, A. F.},
   author={Carne, T. K.},
   author={Minda, D.},
   author={Ng, T. W.},
   title={Random iteration of analytic maps},
   journal={Ergodic Theory Dynam. Systems},
   volume={24},
   date={2004},
   number={3},
   pages={659--675},
}

\bib{BeMi2007}{article}{
   author={Beardon, A. F.},
   author={Minda, D.},
   title={The hyperbolic metric and geometric function theory},
   conference={
      title={Quasiconformal mappings and their applications},
   },
   book={
      publisher={Narosa, New Delhi},
   },
   date={2007},
   pages={9--56},
}

\bib{BeEvFaRiSt2022}{article}{
   author={Benini, A. M.},
   author={Evdoridou, V.},
   author={Fagella, N.},
   author={Rippon, P. J.},
   author={Stallard, G. M.},
   title={Classifying simply connected wandering domains},
   journal={Math. Ann.},
   volume={383},
   date={2022},
   number={3-4},
   pages={1127--1178},
}

\bib{BeEvFaRiSt2024}{article}{
   author={Benini, Anna Miriam},
   author={Evdoridou, Vasiliki},
   author={Fagella, N\'uria},
   author={Rippon, Philip J.},
   author={Stallard, Gwyneth M.},
   title={Boundary dynamics for holomorphic sequences, non-autonomous
   dynamical systems and wandering domains},
   journal={Adv. Math.},
   volume={446},
   date={2024},
   pages={Paper No. 109673, 51},
}

\bib{BrKrRo2023}{article}{
   author={Bracci, Filippo},
   author={Kraus, Daniela},
   author={Roth, Oliver},
   title={A new Schwarz-Pick lemma at the boundary and rigidity of holomorphic maps},
   journal={Adv. Math.},
   volume={432},
   date={2023},
   pages={109--262},
}



\bib{Fe2022}{article}{
   author={Ferreira, Gustavo Rodrigues},
   title={Multiply connected wandering domains of meromorphic functions:
   internal dynamics and connectivity},
   journal={J. Lond. Math. Soc. (2)},
   volume={106},
   date={2022},
   number={3},
   pages={1897--1919},
}

\bib{Fe2023}{article}{
   author={Ferreira, G. R.},
   title={A note on forward iteration of inner functions},
   journal={Bull. Lond. Math. Soc.},
   volume={55},
   date={2023},
   number={3},
   pages={1143--1153},
}


\bib{GoKa2020}{article}{
   author={Gou\"ezel, S\'ebastien},
   author={Karlsson, Anders},
   title={Subadditive and multiplicative ergodic theorems},
   journal={J. Eur. Math. Soc. (JEMS)},
   volume={22},
   date={2020},
   number={6},
   pages={1893--1915},
}

\bib{GuKoMoRo2025}{article}{
   author={Gumenyuk, Pavel},
   author={Kourou, Maria},
   author={Moucha, Annika},
   author={Roth, Oliver},
   title={Hyperbolic distortion and conformality at the boundary},
   journal={Adv. Math.},
   volume={470},
   date={2025},
   pages={Paper No. 110251, 52},
}

\bib{He1941}{article}{
   author={Heins, M. H.},
   title={A generalization of the Aumann-Carath\'{e}odory ``Starrheitssatz.''},
   journal={Duke Math. J.},
   volume={8},
   date={1941},
   pages={312--316},
}

\bib{Heins1941}{article}{
   author={Heins, M. H.},
   title={On the iteration of functions which are analytic and single-valued in a given multiply-connected region},
   journal={Am. J. Math.},
   volume={63},
   date={1941},
   pages={461-480},
}


\bib{He1991}{article}{
   author={Heins, M. H.},
   title={Some results concerning the iteration of analytic functions mapping the open unit disk into itself},
   status={unpublished},
   date={1991}
}

\bib{Hu1981}{article}{
   author={Hutchinson, John E.},
   title={Fractals and self-similarity},
   journal={Indiana Univ. Math. J.},
   volume={30},
   date={1981},
   number={5},
   pages={713--747}
 }

\bib{JaSh2022}{article}{
   author={Jacques, Matthew},
   author={Short, Ian},
   title={Semigroups of isometries of the hyperbolic plane},
   journal={Int. Math. Res. Not. IMRN},
   date={2022},
   number={9},
   pages={6403--6463},
}

\bib{KeLa2007}{book}{
   author={Keen, Linda},
   author={Lakic, Nikola},
   title={Hyperbolic geometry from a local viewpoint},
   series={London Mathematical Society Student Texts},
   volume={68},
   publisher={Cambridge University Press, Cambridge},
   date={2007},
   pages={x+271},
}

\bib{Ke1955}{book}{
   author={Kelley, J. L.},
   title={General topology},
   publisher={D. Van Nostrand Co., Inc., Toronto-New York-London},
   date={1955},
   pages={xiv+298},
}

\bib{Ku2007}{article}{
   author={Kuznetsov, A.},
   title={Semi-groups of analytic functions that contain the identity map},
   journal={Comput. Methods Funct. Theory},
   volume={7},
   date={2007},
   number={1},
   pages={239--247},
   issn={1617-9447},
}




\end{biblist}
\end{bibdiv}
\end{document}